\documentclass[12pt]{amsart}
\usepackage{amsfonts,amsthm,amsmath}
\usepackage{url}
\newtheorem{thm}{Theorem}
\newtheorem{prop}[thm]{Proposition}
\newtheorem{lem}[thm]{Lemma}

\theoremstyle{remark}
\newtheorem{rem}[thm]{Remark}

\newcommand{\FF}{\mathbb{F}}

\newcommand{\0}{\mathbf{0}}
\newcommand{\cC}{\mathcal{C}}

\DeclareMathOperator{\wt}{wt}

\begin{document}

\title[Ternary codes of length $12$]
{On the classification of certain ternary codes of length $12$}

\author{Makoto Araya}
\address{
Department of Computer Science, Faculty of Informatics,
Shizuoka University,
Hamamatsu 432--8011, Japan}
\email{araya@inf.shizuoka.ac.jp}
\author{Masaaki Harada}
\address{
Research Center for Pure and Applied Mathematics, 
Graduate School of Information Sciences, 
Tohoku University, Sendai 980--8579, Japan}
\email{mharada@m.tohoku.ac.jp}


\begin{abstract}
Shimada and Zhang studied the existence of
polarizations on some supersingular $K3$ surfaces
by reducing the existence of the polarizations 
to that of ternary $[12,5]$ codes 
satisfying certain conditions.
In this note, we give a classification
of ternary $[12,5]$ codes satisfying the conditions.
To do this,
ternary $[10,5]$ codes are classified for 
minimum weights $3$ and $4$.
\end{abstract}

\maketitle

\section{Introduction}\label{Sec:1}

A ternary $[n,k]$ code $C$ is a $k$-dimensional vector subspace
of $\FF_3^n$,
where $\FF_3$ denotes the finite field of order $3$.
The weight $\wt(x)$ of a vector $x$ is
the number of non-zero components of $x$.
The minimum non-zero weight of all codewords in $C$ is called
the minimum weight of $C$.
A ternary $[n,k,d]$ code is a ternary $[n,k]$ code with minimum weight $d$.
Throughout this note, 
we denote the minimum weight of a code $C$ by $d(C)$.

Shimada and Zhang~\cite{Shimada} studied 
the existence of
polarizations on 
the supersingular $K3$ surfaces in characteristic $3$
with Artin invariant $1$
(see~\cite[Theorem~1.5]{Shimada} for the details).
This was done by
reducing the problem of the existence of
the polarizations 
to a problem of the existence of ternary $[12,5]$ codes $C$ 
satisfying the following
conditions:
\begin{align}
\label{C1}
& \wt((x_1,x_2,\ldots,x_{10}))\equiv y_1y_2 \pmod 3,
\\ 
\label{C2}
& 
\text{ if } c 
\text{ is not the zero vector, } \text{ then }
\wt((x_1,x_2,\ldots,x_{10})) \ge 3,
\\ 
\label{C3}
&
\text{ if }\wt((x_1,x_2,\ldots,x_{10}))=3,
\text{ then } (y_1,y_2) \ne (0,0), 
\end{align}
for any codeword 
$c=(x_1,x_2,\ldots,x_{10},y_1,y_2) \in C$ (see~\cite[Claim~5.2]{Shimada}).
Seven ternary $[12,5]$ codes satisfying the conditions
(\ref{C1})--(\ref{C3}) were found by
Shimada and Zhang~\cite{Shimada}.
This motivates us to classify all such ternary $[12,5]$ codes.

For ternary $[12,5]$ codes satisfying the conditions
(\ref{C1})--(\ref{C3}),
the following equivalence is considered in~\cite{Shimada}.
We say that two ternary $[12,5]$ codes satisfying the conditions
(\ref{C1})--(\ref{C3})
are {\em SZ-equivalent} if one can be obtained from the
other by using the following:
\begin{equation}\label{eq:eq}
\begin{split}
&(x_1,\ldots,x_{10},y_1,y_2) 
\\
&\mapsto
((-1)^{\alpha_1} x_{\sigma(1)},\ldots,
(-1)^{\alpha_{10}} x_{\sigma(10)},
(-1)^\beta y_{\tau(1)},(-1)^\beta y_{\tau(2)}),
\end{split}
\end{equation}
where $\alpha_1,\ldots,\alpha_{10},\beta \in \{0,1\}$
and $\sigma \in S_{10}, \tau \in S_2$
(see~\cite[Remark~5.3]{Shimada}).
Here, $S_n$ denotes the symmetric group of degree $n$.

The main aim of this note is to give the following
classification, which is based on a computer calculation.

\begin{thm}\label{thm}
Any ternary $[12,5]$ code satisfying the conditions
{\rm (\ref{C1})--(\ref{C3})} is SZ-equivalent to
one of the seven codes given in {\rm \cite[Remark~5.3]{Shimada}}.
\end{thm}

To complete the above classification,
ternary $[10,5,d]$ codes are classified for 
the cases $d=3$ and $4$.

\section{Characterization of ternary $[12,5]$ codes 
satisfying (\ref{C1})--(\ref{C3})}\label{Sec:pre}

Let $C$ be a ternary $[n,k]$ code.
The code obtained from $C$ by deleting some coordinates $I$
in each codeword is called the {\em punctured code} of $C$ on $I$.
Throughout this note, we denote
the punctured code of a ternary  $[12,5]$ code
$C$ on $\{11,12\}$ by $Pun(C)$.
Let $d_{max}(n,k)$ denote the largest minimum weight among
ternary $[n,k]$ codes.
It is known that $d_{max}(10,5)=5$ and 
$d_{max}(12,5)=6$ (see~\cite{Brouwer-Handbook}, \cite{Grassl}).

\begin{lem}\label{lem:pun}
If $C$ is a ternary  $[12,5]$ code satisfying the condition
{\rm (\ref{C2})},
then $Pun(C)$ is a
ternary $[10,5]$ code and $d(Pun(C)) \in \{3,4,5\}$.
\end{lem}
\begin{proof}
Suppose that $Pun(C)$ has dimension at most $4$.
Then we may assume without loss of generality that
$C$ has generator matrix whose first row is 
$(0,0,\ldots,0,y_1,y_2)$, where
$(y_1,y_2) \ne (0,0)$.
This contradicts with the condition (\ref{C2}).
Hence, $Pun(C)$ is a ternary $[10,5]$ code.
Again, by the condition (\ref{C2}), 
$Pun(C)$ has minimum weight at least $3$.
Since $d_{max}(10,5)=5$, the result follows.
\end{proof}

\begin{lem}\label{lem:d}
Let $C$ be a ternary  $[12,5]$ code satisfying the conditions
{\rm (\ref{C1})--(\ref{C3})}.
\begin{itemize}
\item[\rm (i)]   $d(Pun(C))\in \{4,5\}$ if and only if $d(C)=6$.
\item[\rm (ii)]  $d(Pun(C))=3$ if and only if $d(C) =4$.
\end{itemize}
\end{lem}
\begin{proof}
By Lemma~\ref{lem:pun},
$Pun(C)$ is a
ternary $[10,5]$ code and $d(Pun(C)) \in \{3,4,5\}$.
It is trivial that $d(C)-d(Pun(C)) \in \{0,1,2\}$.

Suppose that $d(Pun(C))\in \{4,5\}$.
Let $x=(x_1,\ldots,x_{10})$ be a codeword of $Pun(C)$.
If $\wt(x)=4$ (resp.\ $5$), then any corresponding codeword  
$(x_1,\ldots,x_{10},y_1,y_2)$ of $C$
has weight $6$ (resp.\ $7$), by the condition (\ref{C1}).
Since $d_{max}(12,5)=6$, we have that $d(C)=6$.
Conversely, if $d(C)=6$, then it follows from 
$d_{max}(10,5)=5$ that $d(Pun(C))\in \{4,5\}$.

Suppose that $d(Pun(C))=3$.
Let $x=(x_1,\ldots,x_{10})$ be a codeword of $Pun(C)$.
If $\wt(x)=3$, then any corresponding codeword 
$(x_1,\ldots,x_{10},y_1,y_2)$ of $C$
has weight $4$, by the conditions (\ref{C1}) and (\ref{C3}).
Hence, we have that  $d(C) =4$.
Conversely, suppose that $d(C)=4$.  
Then $d(Pun(C)) \in \{2,3,4\}$.  By the condition (\ref{C2}),
$d(Pun(C)) \in \{3,4\}$.  From the statement (i),
$d(Pun(C)) = 3$. 
\end{proof}

Recall that
two ternary codes are equivalent if one can be obtained from the
other by permuting the coordinates and (if necessary) changing
the signs of certain coordinates.
For ternary $[10,5]$ codes, we consider this
usual equivalence.

\begin{lem}\label{lem:puneq}
Let $C$ and $C'$ be ternary  $[12,5]$ 
codes satisfying the conditions
{\rm (\ref{C1})--(\ref{C3})}.
Suppose that 
$C$ and $C'$ are SZ-equivalent.
Then $Pun(C)$ and $Pun(C')$  are equivalent.
\end{lem}
\begin{proof}
Suppose that $C$ is obtained from $C'$
by (\ref{eq:eq}).
Then $Pun(C)$ can be obtained from  $Pun(C')$ by 
$
(x_1,\ldots,x_{10}) \mapsto
((-1)^{\alpha_1} x_{\sigma(1)},\ldots,
(-1)^{\alpha_{10}} x_{\sigma(10)}).
$
\end{proof}

By considering the inverse operation of puncturing,
one can construct ternary  $[12,5]$ 
codes satisfying the conditions
(\ref{C1})--(\ref{C3}) as follows.
Throughout this note, we denote the ternary code 
having generator matrix $G$ by $C(G)$.
Suppose that $C(G)$ is a 
ternary $[10,5]$ code and $d(C(G)) \in \{3,4,5\}$.
Let $g_i$ denote the $i$th row of $G$.
Consider the following generator matrix:
\begin{equation}\label{eq:gmat}
\left(\begin{array}{ccccc|cc}
 &    &  &  & & a_1 & b_1\\
 &    & G&  & & \vdots & \vdots\\
 &    &  &  & & a_5 & b_5\\
\end{array}\right),
\end{equation}
where
\[
(a_i,b_i)=
\begin{cases}
(0,0),(0,1),(0,2),(1,0),(2,0) &\text{ if } \wt(g_i)\equiv 0 \pmod 3,\\
(1,1),(2,2)                   &\text{ if } \wt(g_i)\equiv 1 \pmod 3,\\
(1,2),(2,1)                   &\text{ if } \wt(g_i)\equiv 2 \pmod 3.
\end{cases}
\]
We denote this generator matrix
by $G(a,b)$, where $a=(a_1,\ldots,a_5)$ and $b=(b_1,\ldots,b_5)$.
The set of the codes $C(G(a,b))$ contains all
ternary $[12,5]$ codes $C$ satisfying the conditions
(\ref{C1}) and $Pun(C(G(a,b)))=C(G)$.
Hence, in this way, every 
ternary  $[12,5]$ code satisfying the conditions
(\ref{C1})--(\ref{C3}) can be obtained from some ternary 
$[10,5]$ code.
Here, by Lemma~\ref{lem:d}, its minimum weight is $3,4$ or $5$.
In addition, if $C(G)$ and $C(G')$ are equivalent $[10,5]$
codes, then
the sets of all codes $C(G(a,b))$ 
satisfying the conditions (\ref{C1})--(\ref{C3}) 
is obtained from the set of all codes $C(G'(a,b))$
satisfying the same conditions
by considering (\ref{eq:eq}) with $\beta=0$ and $\tau$ is the identity
permutation.
Hence, it is sufficient to consider only inequivalent 
ternary $[10,5,d]$ codes with $d \in \{3,4,5\}$
for the classification of 
ternary  $[12,5]$ codes satisfying the conditions
(\ref{C1})--(\ref{C3}).
This is a reason why we consider
the classification of 
ternary $[10,5,d]$ codes with $d \in \{3,4,5\}$
in the next section.

\section{Ternary $[10,5,d]$ codes with $d \in \{3,4,5\}$}
\label{sec:10}

There is a unique ternary $[10,5,5]$ code, up to equivalence~\cite{GS}.
In this section, we give a classification of 
ternary $[10,5,d]$ codes with 
$d \in \{3,4\}$, which is based on a computer calculation.

We describe how ternary $[10,5,3]$ codes and $[10,5,4]$ codes
were classified.  
Let $C$ be a ternary $[10,5,3]$ code
(resp.\ $[10,5,4]$ code).
We may assume without loss of generality that
$C$ has generator matrix 
of the following form:
\[
G=
\left(\begin{array}{cccccc}
 &I_{5}&  & A &\\
\end{array}\right),
\]
where $A$ is a $5 \times 5$ matrix over $\FF_3$
and $I_5$ denotes the identity matrix of order $5$.
Thus, we only need consider the set of $A$,
rather than the set of generator matrices.
The set of matrices $A$ was constructed, row by row, 
as follows, by a computer calculation.
Let $r_i$ be the $i$th row of $A$.
Then, we may assume without loss of generality that
$r_1=(0,0,0,1,1)$ (resp.\ $r_1=(0,0,1,1,1)$), 
by permuting and (if necessary) changing
the signs of the columns of $A$.

Let $e_1,\ldots,e_5$ denote the vectors $(1,0,0,0,0),\ldots,
(0,0,0,0,1)$, respectively.
We denote the ternary code
generated by vectors $y_1,y_2,\ldots,y_s$
by $\langle y_1,y_2,\ldots,y_s \rangle$.
For $x=(x_1,\ldots,x_5) \in \FF_3^5$, consider
the following conditions:
\begin{itemize}
\item
the first nonzero element of $x$ is $1$,

\item
$\wt(x)\ge 2$ (resp.\ $\wt(x)\ge 3$),

\item
the ternary code $\langle (e_1,r_1),(e_2,x) \rangle$
has minimum weight $3$ (resp.\ $4$),

\item
$x_{1} \le x_{2} \le x_{3} \le 1$ and $x_{4} \le x_{5}$
(resp.\ 
$x_{1} \le x_{2} \le 1$ and $x_{3} \le x_{4} \le x_{5}$),
where we consider a natural order on 
the elements of $\FF_3=\{0,1,2\}$ by $0<1<2$.
\end{itemize}
The determination of the minimum weights
was done by a computer calculation for all codes in this note.
Let $X_1$ be the set of vectors $x \in \FF_3^5$ 
satisfying the first three conditions.
Let $X_2$ be the set of vectors $x \in X_1$ 
satisfying the fourth condition.
Our computer calculation shows that
$(\#X_1,\#X_2) = (115,18)$
(resp.\ $(88,14)$).
Define a lexicographical order on 
$X_1$ induced by the above order of $\FF_3$,
that is, $(a_1,\dots,a_5) < (b_1,\dots,b_5)$ 
if $a_1 < b_1$, or 
$a_1=b_1,\dots,a_{k}=b_{k}$ and $a_{k+1}<b_{k+1}$ 
for some $k \in \{1,2,3,4\}$.
The matrices $A$ were constructed, row by row, 
satisfying 
the following conditions:
\begin{itemize}
\item 
the ternary code
$\langle (e_s,r_s) \mid s=1,2,3 \rangle$ 
has minimum weight $3$ (resp.\ $4$), where 
$r_2 \in X_2,\ r_3 \in X_1$,

\item 
the ternary code
$\langle (e_s,r_s) \mid s=1,2,3,4 \rangle$ 
has minimum weight $3$ (resp.\ $4$), where 
$r_2 \in X_2,\ r_3,r_4 \in X_1\ (r_3 < r_4)$,

\item 
the ternary code
$\langle (e_s,r_s) \mid s=1,2,3,4,5 \rangle$
has minimum weight $3$ (resp.\ $4$), where 
$r_2 \in X_2,\ 
r_3, r_4, r_5 \in X_1\ (r_3<r_4<r_5)$.
\end{itemize}
It is obvious that the set of the matrices $A$
which must be checked to achieve
a complete classification, can be obtained in this way.

Then, by a computer calculation,
we found $4328352$ (resp.\ $650051$) matrices $A$.
Our computer calculation shows the $4328352$ ternary $[10,5,3]$ codes 
(resp.\ $650051$ ternary $[10,5,4]$ codes)
are divided into $527$ (resp.\ $64$)
classes by comparing their Hamming weight enumerators.
For each Hamming weight enumerator,
to test equivalence of codes, we use the
algorithm given in~\cite[Section 7.3.3]{KO} as follows.
For a ternary $[n,k]$ code $C$, define 
the digraph $\Gamma(C)$ with vertex set 
\[
(C-\{\0\}) \cup (\{1,2,\dots,n\}\times (\FF_3-\{0\})) 
\]
and arc set 
\begin{align*}
&
\{(c,(j,c_j)) \mid c=(c_{1},\ldots,c_{n}) \in C-\{\0\},  c_j \ne 0, 
1 \le j \le n\} 
\\ &
\cup \{((j,1),(j,2)),((j,2),(j,1)) \mid 1 \le j \le n\}.
\end{align*}
Then, two ternary $[n,k]$ codes $C$ and $C'$ are equivalent
if and only if $\Gamma(C)$ and $\Gamma(C')$  are isomorphic.
We use the package {\tt GRAPE}~\cite{GRAPE}
of {\tt GAP}~\cite{GAP4} for digraph isomorphism testing.
After checking whether codes are equivalent or not 
by a computer calculation for each
Hamming weight enumerator, we have the following:

\begin{prop}
There are $135$ ternary $[10,5,4]$ codes, up to equivalence.
There are $1303$ ternary $[10,5,3]$ codes, up to equivalence.
\end{prop}

We denote the $135$ ternary $[10,5,4]$ codes by
$C_{10,4,i}$ $(i=1,2,\ldots,135)$, and 
we denote the $1303$ ternary $[10,5,3]$ codes by
$C_{10,3,i}$ $(i=1,2,\ldots,1303)$.
Generator matrices of all codes 
can be obtained electronically from~\cite{Araya}.

The unique ternary $[10,5,5]$ code $C_{10,5}$  is 
formally self-dual,
that is, the Hamming weight enumerators of the code and its
dual code are identical.
In addition, the supports of the codewords of minimum weight in 
$C_{10,5}$  form a $3$-design~\cite{DGH}.
We verified by a computer calculation that 
$38$ ternary $[10,5,4]$ codes and 
$242$ ternary $[10,5,3]$ codes are formally self-dual.
In addition, we verified by a computer calculation that 
the supports of the codewords of minimum weight in 
only the code $C_{10,4,132}$ form a $2$-design and
the supports of the codewords of minimum weight in 
$C_{10,4,i}$ form a $1$-design
for only $i=6,86,87,89,132$.

\section{Ternary  $[12,5]$ codes satisfying (\ref{C1})--(\ref{C3})}

In this section,
we give a classification of
ternary $[12,5]$ codes satisfying the conditions
{\rm (\ref{C1})--(\ref{C3})}, which is based on a computer calculation.
This is obtained from the classification of 
ternary $[10,5,d]$ codes with $d \in \{3,4,5\}$,
by using the method given in Section~\ref{Sec:pre}.

\subsection{From the $[10,5,5]$ code and the $[10,5,4]$ codes}

As described in the previous section,
there is a unique ternary $[10,5,5]$ code, up to equivalence~\cite{GS}.
It follows from~\cite{DGH} that
this code $C_{10,5}$ has generator matrix 
$G_{10,5}=
\left(\begin{array}{cccccc}
I_{5} &  A \\
\end{array}\right)$,
where $A$ is the following circulant matrix:
\[
A=
\left(\begin{array}{cccc}
12210\\
01221\\
10122\\
21012\\
22101\\
\end{array}\right).
\]
In order to construct all ternary $[12,5]$ codes $C$ 
satisfying the conditions (\ref{C1}) and $Pun(C)=C_{10,5}$,
we consider generator matrices $G_{10,5}(a,b)$ of the form (\ref{eq:gmat}).
Since the weight of each row of $G_{10,5}$ is $5$, 
$(a_i,b_i)=(1,2)$ or $(2,1)$ for $i=1,2,3,4,5$.
By (\ref{eq:eq}), we may assume that $(a_1,b_1)=(1,2)$.
Since the weight of the sum of the first row and the second row
of $G_{10,5}$ is $5$, $(a_2,b_2)$ must be $(1,2)$.
Similarly, we have that 
$(a_i,b_i)=(1,2)$ for $i=3,4,5$, 
since $A$ is circulant.
In addition, we verified by a computer calculation that this code 
satisfies the condition {\rm (\ref{C1})}.
Note that the code automatically satisfies 
the conditions (\ref{C2}) and (\ref{C3}).
We denote the code by $\cC_{12,1}$.

Now, consider the ternary $[10,5,4]$ codes $C_{10,4,i}$ $(i=1,2,\ldots,135)$.
By considering generator matrices of the form (\ref{eq:gmat}),
we found all ternary $[12,5]$ codes $C$ 
satisfying the conditions (\ref{C1}) and $Pun(C)=C_{10,4,i}$.
This was done by a computer calculation.
We denote by $G_{10,4,i}$ the generator matrix 
$\left(\begin{array}{cccccc}
 I_{5}&  A \\
\end{array}\right)$
of $C_{10,4,i}$ for each $i$.
Since the weight of the first row of $A$ is $3$ (see Section~\ref{sec:10}),
by (\ref{eq:eq}), we may assume that $(a_1,b_1)=(1,1)$ in (\ref{eq:gmat}).
Under this situation, 
we verified by a computer calculation that only the codes
$C_{10,4,60}$ and $C_{10,4,132}$ give
ternary $[12,5]$ codes satisfying the condition (\ref{C1}).
Note that these codes automatically satisfy 
the conditions (\ref{C2}) and (\ref{C3}).
In Table~\ref{Tab:d4mat}, we list the matrices $A$ and 
$(a^T,b^T)$ in $G_{10,4,i}(a,b)$ for $i=60,132$,
where $a^T$ denotes the transposed of a vector $a$.
It can be seen by hand that 
the two codes $C(G_{10,4,60}(a,b))$ are SZ-equivalent.
By Lemma~\ref{lem:puneq},
there are two ternary $[12,5]$ codes $C$
satisfying the conditions {\rm (\ref{C1})--(\ref{C3})}
and the condition that $Pun(C)$ is a ternary $[10,5,4]$ code.
We denote the two codes by $\cC_{12,2}$ and $\cC_{12,3}$,
respectively (note that take the first $(a^T,b^T)$ for $i=60$).

\begin{table}[thb]
\caption{Generator matrices $G_{10,4,i}(a,b)$  $(i=60,132)$}
\label{Tab:d4mat}
\begin{center}
{\footnotesize
\begin{tabular}{c|c|cc}
\noalign{\hrule height0.8pt}
$i$& $A$ & \multicolumn{2}{c}{$(a^T,b^T)$} \\
\hline
&&&\\
60 &
$\left(\begin{array}{c}
0 0 1 1 1\\
0 1 0 1 1\\
1 0 1 0 1\\
1 1 0 0 1\\
1 2 2 1 0
\end{array}\right)$&$\left(\begin{array}{c}
1 1\\
2 2\\
2 2\\
1 1\\
1 2\\
\end{array}\right),$&$\left(\begin{array}{c}
1 1\\
2 2\\
2 2\\
1 1\\
2 1\\
\end{array}\right)$ \\ 
&&&\\
\hline
&&&\\
132 &
$\left(\begin{array}{c}
0 0 1 1 1\\
0 1 0 1 1\\
1 0 1 0 1\\
1 1 0 0 1\\
1 1 1 1 1
\end{array}\right)$&$\left(\begin{array}{c}
1 1\\
2 2\\
2 2\\
1 1\\
0 0\\
\end{array}\right)$ \\
&&&\\
\noalign{\hrule height0.8pt}
\end{tabular}
}
\end{center}
\end{table}

Lemma~\ref{lem:d} shows that there are no other 
ternary $[12,5,6]$ codes
satisfying the conditions {\rm (\ref{C1})--(\ref{C3})}.
Hence, we have the following:

\begin{lem}\label{lem:d6}
Up to SZ-equivalence,
there are three ternary $[12,5,6]$ codes
satisfying the conditions {\rm (\ref{C1})--(\ref{C3})}.
\end{lem}

\subsection{From the $[10,5,3]$ codes}

By considering generator matrices of the form (\ref{eq:gmat}),
we found all ternary $[12,5]$ codes $C$ 
satisfying the conditions (\ref{C1}) and 
$Pun(C)=C_{10,3,i}$ $(i=1,2,\ldots,1303)$.
This was done by a computer calculation.
We denote by $G_{10,3,i}$ the generator matrix 
$\left(\begin{array}{cccccc}
 I_{5}&  A \\
\end{array}\right)$
of $C_{10,3,i}$ for each $i$.
Since the weight of the first row of $A$ is $2$ (see Section~\ref{sec:10}),
by (\ref{eq:eq}), we may assume that $(a_1,b_1)=(0,1)$ in (\ref{eq:gmat}).
Under this situation, 
we verified by a computer calculation that only the codes
$C_{10,3,i}$ give
ternary $[12,5]$ codes satisfying the condition (\ref{C1})
for
\begin{align*}
i=&302, 639, 662, 666, 667, 756, 878, 957, 958, 987, \\
&1210, 1215, 1241, 1245, 1263, 1285, 1297, 1298, 1299.
\end{align*}
In this case, there are codes 
satisfying the condition (\ref{C1}), but not (\ref{C3}). 
We verified by a computer calculation that only the codes
$C_{10,3,i}$ give
ternary $[12,5]$ codes satisfying the conditions 
(\ref{C1}) and (\ref{C3}) for $i=302,666,987,1245$.
Note that these four codes automatically satisfy the condition (\ref{C2}).
In Table~\ref{Tab:d3mat}, we list the matrices $A$ and 
$(a^T,b^T)$ in $G_{10,4,i}(a,b)$ for $i=302,666,987,1245$.
By Lemma~\ref{lem:puneq},
there are four ternary $[12,5]$ codes
satisfying the conditions {\rm (\ref{C1})--(\ref{C3})}
and the condition that $Pun(C)$ is a ternary $[10,5,3]$ code.
We denote the four codes by $\cC_{12,i}$ $(i=4,5,6,7)$,
respectively.

\begin{table}[thb]
\caption{Generator matrices $G_{10,3,i}(a,b)$ $(i=302,666,987,1245)$}
\label{Tab:d3mat}
\begin{center}
{\footnotesize
\begin{tabular}{c|c|c||c|c|c}
\noalign{\hrule height0.8pt}
$i$& $A$ &  $(a^T,b^T)$ &
$i$& $A$ &  $(a^T,b^T)$ \\
\hline
&&&&&\\
302 &
$\left(\begin{array}{c}
0 0 0 1 1\\
0 1 1 0 0\\
1 0 1 0 1\\
1 1 0 1 0\\
1 1 2 2 1
\end{array}\right)$&$\left(\begin{array}{c}
0 1\\0 1\\2 2\\2 2\\0 1\\
\end{array}\right)$ &
987&
$\left(\begin{array}{c}
0 0 0 1 1\\
0 0 1 0 1\\
0 1 0 1 0\\
0 1 1 0 0\\
1 1 2 2 1
\end{array}\right)$&$\left(\begin{array}{c}
0 1\\2 0\\2 0\\0 1\\0 0\\
\end{array}\right)$ \\ 
&&&&&\\
\hline
&&&&&\\
666 &
$\left(\begin{array}{c}
0 0 0 1 1\\
0 1 1 0 0\\
1 0 1 0 1\\
1 1 0 1 0\\
1 2 2 2 2
\end{array}\right)$&$\left(\begin{array}{c}
0 1\\0 1\\2 2\\2 2\\2 0\\
\end{array}\right)$ &
1245&
$\left(\begin{array}{c}
0 0 0 1 1\\
0 0 1 0 1\\
0 1 0 1 0\\
0 1 1 0 0\\
0 1 1 1 1
\end{array}\right)$&$\left(\begin{array}{c}
0 1\\2 0\\2 0\\0 1\\1 2\\
\end{array}\right)$ \\
&&&&&\\
\noalign{\hrule height0.8pt}
\end{tabular}
}
\end{center}
\end{table}

Lemma~\ref{lem:d} shows that there are no other 
ternary $[12,5,4]$ codes
satisfying the conditions {\rm (\ref{C1})--(\ref{C3})}.
Hence, we have the following:

\begin{lem}\label{lem:d5}
Up to SZ-equivalence,
there are four ternary $[12,5,4]$ codes
satisfying the conditions {\rm (\ref{C1})--(\ref{C3})}.
\end{lem}

Up to SZ-equivalence,
seven ternary $[12,5]$ codes 
satisfying the conditions (\ref{C1})--(\ref{C3}) are known
(see~\cite[Remark~5.3]{Shimada}).
Lemmas~\ref{lem:d6} and~\ref{lem:d5} show that there are no other 
ternary $[12,5]$ codes
satisfying the conditions {\rm (\ref{C1})--(\ref{C3})}.
Therefore, we have Theorem~\ref{thm}.

\subsection{Some properties}

For the ternary $[12,5]$ codes $C$
satisfying the conditions {\rm (\ref{C1})--(\ref{C3})},
instead of the Hamming weight enumerators, 
we consider the weight enumerators
$\displaystyle{
\sum_{(x_1,\ldots,x_{10},y_1,y_2) \in C} x^{\wt((x_1,\ldots,x_{10}))}
y^{n_1}z^{n_2}
}$,
where $n_1$ and $n_2$ are the numbers of $1$'s and $2$'s in 
$(y_1,y_2)$, respectively.
We verified by a computer calculation 
that the codes $\cC_{12,i}$ $(i=1,2,\ldots,7)$ have
the following weight enumerators $W_i$:
\begin{align*}
W_1=& 1+ 72 x^5 y z+ 60 x^6 + 90 x^8 y z+ 20 x^9,
\\ 
W_2=&
1+  9x^4 z^2 +  9x^ 4 y^2  + 18x^ 5 y z + 24x^ 6   + 36x^ 6 z + 36x^ 6 y  + 18x^ 7  z^2 
+ 18x^ 7 y^2  
\\ &
+ 36x^ 8 y z +  2x^ 9   + 18x^ 9 z + 18x^ 9 y,
\\ &
1+ 15x^ 4  z^2 + 15x^ 4 y^2  + 60x^ 6   + 60x^ 7  z^2 + 60x^ 7 y^2  + 20x^ 9 + 6x^{10}  z^2 
+  6x^{10} y^2,
\\ &1 +  2x^3z +  2x^3y +  4x^4z^2 +  4x^4y^2 + 24x^5yz + 18x^6 + 38x^6z+ 38x^6y
\\ &
+ 22x^7z^2 + 22x^7y^2 + 30x^8yz +  8x^9 + 14x^9z + 14x^9y
+  x^{10}z^2 +  x^{10}y^2,
\\ &
1+  3x^3z+  3x^3y+  3x^4z^2+  3x^4y^2+ 18x^5yz+ 24x^6+ 39x^6z+ 39x^6y
\\ &
+ 21x^7z^2+ 21x^7y^2+ 36x^8yz+  2x^9+ 12x^9z+ 12x^9y+  3x^{10}z^2+  3x^{10}y^2,
\\ &
1+  4x^3z+  4x^3y+  5x^4z^2+  5x^4y^2+ 24x^5yz+ 18x^6+ 34x^6z+ 34x^6y
\\ &
+ 20x^7z^2+ 20x^7y^2+ 30x^8yz+  8x^9+ 16x^9z+ 16x^9y+  2x^{10}z^2+2x^{10}y^2,
\\ &
1+  6x^3z+  6x^3y+  9x^4z^2+  9x^4y^2+ 36x^5yz+ 24x^6+ 42x^6z+ 42x^6y
\\ &
+ 18x^7z^2+ 18x^7y^2+ 18x^8yz+  2x^9+  6x^9z+  6x^9y,
\end{align*}
respectively.
These weight enumerators guarantee that the codes
$\cC_{12,i}$ $(i=1,2,\ldots,7)$
satisfy the conditions {\rm (\ref{C1})--(\ref{C3})}.
By putting $y=z=1$,
the above weight enumerators
determine the Hamming weight enumerators of $Pun(\cC_{12,i})$ $(i=1,2,\ldots,7)$.
This implies that $\cC_{12,1}$
is SZ-equivalent to 
${\mathcal C}_7$ in~\cite[Table~5.1]{Shimada}.
In addition, 
by comparing generator matrices, it is easy to see that
$\cC_{12,i}$ $(i=2,3,\ldots,7)$ are equal to 
${\mathcal C}_6$,
${\mathcal C}_5$,
${\mathcal C}_3$,
${\mathcal C}_4$,
${\mathcal C}_2$ and 
${\mathcal C}_1$ in~\cite[Table~5.1]{Shimada}, respectively.

\begin{rem}
Shimada and Zhang~\cite{Shimada} also considered 
the existence of ternary $[12,4,6]$ codes 
satisfying the condition that
all codewords have weight divisible by three,
in the proof of Theorem~1.4
(see~\cite[Claim~6.2]{Shimada}).
We point out that a code satisfying the condition
is self-orthogonal.
There is a unique 
self-orthogonal ternary $[12,4,6]$ code,
up to equivalence~\cite[Table~1]{MPS}.
\end{rem}

\bigskip
\noindent {\bf Acknowledgments.}
The authors would like to thank the anonymous referee for useful comments.
This work is supported by JSPS KAKENHI Grant Number 23340021.
In this work, the supercomputer of ACCMS, Kyoto University was partially used.

\end{document}